\pdfoutput=1 
\documentclass[11pt]{amsart}

\usepackage{epsfig,overpic,comment}
\usepackage[usenames,dvipsnames,svgnames,table]{xcolor}
\usepackage[hyphens]{url}
\usepackage[pagebackref,linktocpage=true,colorlinks=true,linkcolor=Blue,citecolor=BrickRed,urlcolor=RoyalBlue]{hyperref}
\usepackage[msc-links,abbrev]{amsrefs}
\usepackage{amsmath,amsthm,amssymb,marginnote}
\usepackage{enumerate}
\usepackage[T1]{fontenc}

\newtheorem{theorem}{Theorem}[section]

\newtheorem{lemma}[theorem]{Lemma}

\theoremstyle{definition}
\newtheorem{note}[theorem]{Note}

\newcommand{\be}{\begin{equation}}
\newcommand{\ee}{\end{equation}}
\newcommand{\ol}{\overline}

\newcommand{\R}{\mathbf{R}}

\newcommand{\C}{\mathcal{C}}
\newcommand{\G}{\mathcal{G}}

\newcommand{\M}{\mathcal{M}}

\renewcommand{\epsilon}{\varepsilon}

\makeatletter
\DeclareFontFamily{U}{tipa}{}
\DeclareFontShape{U}{tipa}{m}{n}{<->tipa10}{}
\newcommand{\arc@char}{{\usefont{U}{tipa}{m}{n}\symbol{62}}}%

\newcommand{\arc}[1]{\mathpalette\arc@arc{#1}}

\newcommand{\arc@arc}[2]{%
  \sbox0{$\m@th#1#2$}%
  \vbox{
    \hbox{\resizebox{\wd0}{\height}{\arc@char}}
    \nointerlineskip
    \box0
  }%
}
\makeatother

\makeatletter
\def\@tocline#1#2#3#4#5#6#7{\relax
  \ifnum #1>\c@tocdepth 
  \else
    \par \addpenalty\@secpenalty\addvspace{#2}%
    \begingroup \hyphenpenalty\@M
    \@ifempty{#4}{%
      \@tempdima\csname r@tocindent\number#1\endcsname\relax
    }{%
      \@tempdima#4\relax
    }%
    \parindent\z@ \leftskip#3\relax \advance\leftskip\@tempdima\relax
    \rightskip\@pnumwidth plus4em \parfillskip-\@pnumwidth
    #5\leavevmode\hskip-\@tempdima
      \ifcase #1
       \or\or \hskip 1.3em \or \hskip 2em \else \hskip 5em \fi%
      #6\nobreak\relax
    \hfill\hbox to\@pnumwidth{\@tocpagenum{#7}}\par
    \nobreak
    \endgroup
  \fi}
\makeatother

\newcommand{\nocontentsline}[3]{}
\newcommand{\tocless}[2]{\bgroup\let\addcontentsline=\nocontentsline#1{#2}\egroup}

\begin{document}
\setlength{\baselineskip}{1.2\baselineskip}

\title[Minkowski inequality in Cartan-Hadamard manifolds] 
{Minkowski inequality in Cartan-Hadamard manifolds}

\author{Mohammad Ghomi}
\address{School of Mathematics, Georgia Institute of Technology,
Atlanta, GA 30332}
\email{ghomi@math.gatech.edu}
\urladdr{www.math.gatech.edu/~ghomi}

\author{Joel Spruck}
\address{Department of Mathematics, Johns Hopkins University,
 Baltimore, MD 21218}
\email{js@math.jhu.edu}
\urladdr{www.math.jhu.edu/~js}

\vspace*{-0.75in}
\begin{abstract}
Using harmonic mean curvature flow,  we establish a sharp Minkowski type lower bound for total mean curvature of convex surfaces with a given area in Cartan-Hadamard $3$-manifolds.  This inequality also improves the known estimates for total mean curvature in hyperbolic $3$-space. As an application, we obtain a Bonnesen-style  isoperimetric inequality for  surfaces with convex distance function in nonpositively curved $3$-spaces, via monotonicity results for total mean curvature. This connection between the Minkowski and isoperimetric inequalities is extended to
Cartan-Hadamard manifolds of any dimension.
\end{abstract}

\date{\today \,(Last Typeset)}
\thanks{The research of M.G. was supported by NSF grant DMS-2202337 and a Simons Fellowship. The research of J.S. was supported by a Simons Collaboration Grant}

\maketitle


\section{Introduction}\label{sec:minkowski}
Complete simply connected Riemannian spaces of nonpositive  curvature,  or \emph{Cartan-Hadamard manifolds}, form a natural generalization of Euclidean and hyperbolic spaces. A \emph{strictly convex} hypersurface $\Gamma$ of a Cartan-Hadamard space $M$ is a closed  embedded submanifold of codimension one which,  when properly oriented, 
has positive definite second fundamental form $\mathrm{I\!I}_\Gamma$. The \emph{mean curvature}  of $\Gamma$ is  then given by $H:=\textup{trace}(\mathrm{I\!I}_\Gamma)$, and its \emph{total mean curvature} is defined as $
\M(\Gamma):=\int_\Gamma Hd\mu$.  A celebrated  result of Minkowski  \cite{minkowski1903} states that  in Euclidean space $\R^3$
\be\label{eq:minkowski0}
\M(\Gamma)\geq \sqrt{16\pi |\Gamma|},
\ee
where $|\Gamma|$ denotes the area of $\Gamma$, and equality holds only when $\Gamma$ is a sphere. Extension of this inequality to hyperbolic space $\textbf{H}^3$ has been a long standing problem \cite{santalo1963}, which has been intensively studied  \cites{gallego-solanes,natario2015}, specially with the aid of curvature flows \cites{wang-xia2014,ge-wang-wu 2014,andrews-hu-li-2020,scheuer2020} in recent years; however, the sharp inequality remains unknown. Here we generalize Minkowski's inequality to Cartan-Hadamard manifolds via harmonic mean curvature flow. By \emph{smooth} we mean $\C^\infty$, \emph{curvature} means sectional curvature unless specified otherwise, and a \emph{domain} is a connected open set with compact closure.

\begin{theorem}\label{thm:minkowski1}
Let $\Gamma$ be a smooth strictly convex surface in a Cartan-Hadamard 3-manifold $M$ with curvature $K\leq a\leq 0$. Then 
\be\label{eq:minkowski1}
\M(\Gamma)\geq \sqrt{16\pi|\Gamma|-2a|\Gamma|^2},
\ee
with equality only if the domain bounded by $\Gamma$ is isometric to a ball in $\R^3$.
\end{theorem}

Inequality \eqref{eq:minkowski1} appears to be new even in hyperbolic space $\mathbf{H}^3(a)$ of constant curvature $a<0$ \cite[p. 109]{natario2015}, where the previous best estimate was $\M(\Gamma)\geq \sqrt{-a}\,|\Gamma|$ by Gallego-Solanes \cites{gallego-solanes, brendle-wang} (note that in \cite{gallego-solanes}, $H:=\textup{trace}(\mathrm{I\!I}_\Gamma)/(n-1)$).
Santalo asked \cite{santalo1963}, see  \cite[p. 78]{santalo2009}, whether the sharp inequality in $\mathbf{H}^3(a)$ could be
\be\label{eq:minkowski2}
\M(\Gamma)\geq \sqrt{16\pi|\Gamma|-4a|\Gamma|^2},
\ee
as the lower bound would then correspond to the total mean curvature of a sphere with the same area as $\Gamma$; however, an example by Naveira-Solanes  \cite[p. 815]{santalo2009}, see  \cite[p. 109]{natario2015}, shows that \eqref{eq:minkowski2} cannot in general hold. In Note \ref{note:NS}  we will analyze this example to show that \eqref{eq:minkowski1} is not far from optimal. Under the additional hypothesis that $\Gamma$ is \emph{$h$-convex} (or \emph{horo-convex}), i.e., supported at each point by a horosphere,  \eqref{eq:minkowski2} does hold in $\mathbf{H}^3(a)$ \cites{wang-xia2014,ge-wang-wu 2014}. In Note \ref{note:h-convex}  we will discuss a possible improvement of \eqref{eq:minkowski1} in the $h$-convex case, and in Theorem \ref{thm:minkowski2} we extend \eqref{eq:minkowski1} to nonsmooth  surfaces.

Since total mean curvature is the first variation of area, Minkowski's inequality is closely related to isoperimetric problems in Euclidean space \cite[Sec. 7.2]{schneider2014} \cite[p. 1191]{osserman1978}. Here we apply the extension of \eqref{eq:minkowski1} to nonsmooth  surfaces, together with recent monotonicity results for mean curvature \cite{ghomi-spruck2023},   to establish an isoperimetric inequality in Cartan-Hadamard manifolds in the style of Bonnesen \cite{osserman1979}. This gives a refinement, for convex surfaces, of a theorem of Kleiner \cite{kleiner1992} who first generalized the isoperimetric inequality to $3$-dimensional Cartan-Hadamard manifolds. 
The \emph{inradius}, $\textup{inrad}(\Omega)$, of a domain $\Omega\subset M$ is the supremum of radii of spheres which are contained in $\Omega$. A closed embedded hypersurface $\Gamma$, bounding a domain $\Omega$ in a Cartan-Hadamard manifold, is \emph{$d$-convex} (or \emph{distance-convex}) provided that its distance function is convex on $\Omega$. This condition is weaker than $h$-convexity \cite[Sec. 3]{ghomi-spruck2022}. We let $|\Omega|$ denote volume of $\Omega$.

\begin{theorem}\label{thm:isoperimetric}
Let $\Gamma$ be a smooth  $d$-convex surface in a Cartan-Hadamard 3-manifold, and $\Omega$ be the domain bounded by $\Gamma$. Then
\be\label{eq:bonnesen}
|\Omega|\leq\frac{4\pi}{3}\left(\left(\frac{|\Gamma|}{4\pi}\right)^\frac{3}{2}-
\left(\sqrt{\frac{|\Gamma|}{4\pi}}-\textup{inrad}(\Omega)\right)^3\right),
\ee
with equality only if $\Omega$ is isometric to a ball in $\R^3$.
\end{theorem}

The isoperimetric inequality has been established in Cartan-Hadamard manifolds only up to dimension $4$ \cites{weil1926, croke1984, kleiner1992}, and a Bonnesen-style inequality was also  established recently in dimension $2$ \cite{hoisington-mcgrath2022}. The \emph{Cartan-Hadamard conjecture} states that the isoperimetric inequality should hold in all dimensions  \cites{kloeckner-kuperberg2019,ghomi-spruck2022}.  Kleiner's approach to this problem was based on estimating the total Gauss-Kronecker curvature, which was further studied in \cite{ghomi-spruck2022}; see also \cite{schulze2008,choe-ritore2007} and \cite[Sec. 3.3.2]{ritore2010} for other proofs or variations in dimension 3. Proof of Theorem \ref{thm:isoperimetric} provides an alternative approach based on total mean curvature. An advantage of this approach is that mean curvature satisfies a  monotonicity property, see Lemma \ref{lem:monotonicity}, whereas  Gauss-Kronecker curvature does not \cite{dekster1981}.
In Section \ref{sec:high-dim} below we will show that this method may be deployed in any dimension, as stated in Theorem \ref{thm:higher}.

Minkowski's inequality in Euclidean space is a special case of Alexandrov-Fenchel inequalities for generalized mean curvatures of convex bodies, which may be proved via  Brunn-Minkowski theory of mixed volumes; see \cite[Thm. 7.2.1 \& Note 2, p. 387]{schneider2014}, and \cites{shenfeld-vanhandel2019,shenfeld-vanhandel2022} for more recent treatments. Differential geometric proofs using the isoperimetric inequality and Steiner formulas may be found in \cite{natario2015} and \cite[p. 201]{montiel-ros2009}. Ge-Wang-Wu \cite{ge-wang-wu2014} and Wang-Xia \cite{wang-xia2014} extended the inequality to $h$-convex surfaces in hyperbolic space via curvature flows; see also \cites{andrews-hu-li-2020, wei-xiong2015}. In addition, there has been substantial work on weakening the convexity condition \cites{dalphin2016,chang-wang2013,guan-li2009}, and extensions to other spaces \cites{brendle-wang,scheuer2021} due to applications in general relativity \cite{gibbons1997,lima-girao2016}.

\begin{note}\label{note:NS}
Here we examine the Naveira-Solanes example \cite{natario2015} mentioned above to estimate the optimality of \eqref{eq:minkowski1}.  This object is constructed by taking a 
disc $D$ of radius $r$ in a totally geodesic surface in hyperbolic space $\textbf{H}^3(a)$, and letting $\Gamma=\Gamma(\epsilon,r)$ be the outer parallel surface of $D$ at a small distance $\epsilon$. 
We seek the largest value of $\lambda$ so that
$$
\phi_\lambda(\Gamma):=\M(\Gamma)^2-16\pi|\Gamma|+\lambda a|\Gamma|^2
$$
is nonnegative for all $\Gamma$. As $\phi_\lambda$ is invariant under rescaling of the metric, we may assume  for convenience that $a=-1$. Then
 $$
\lim_{\epsilon\to 0} |\Gamma|= 2|D|=4\pi(\cosh(r)-1).
 $$
Note that $\Gamma$ consists of a pair of topological disks parallel to $D$ plus a half tube $T$ about \ $\partial D$. The mean curvature of the disks vanish as $\epsilon\to 0$. On the other hand, $|T|\to|\partial D|\pi\sinh(\epsilon)$ up to first order, since the full tube about $\partial D$ is fibrated by (geodesic) circles of radius $\epsilon$. So $\M(\Gamma)\to\partial |T|/\partial\epsilon=|\partial D|\pi\cosh(\epsilon)$. Thus
$$
\lim_{\epsilon\to 0}\M(\Gamma)=|\partial D|\pi
=2\pi^2\sinh(r),
$$ 
and we conclude that
$$
\lim_{\epsilon\to0} \phi_\lambda(\Gamma)=4\pi^4 \sinh^2(r)-64\pi^2(\cosh(r)-1)-16\pi^2\lambda(\cosh(r)-1)^2.
$$
Setting this quantity $\geq 0$ yields
$$
\lambda \leq \frac{\pi^2\sinh^2(r)-16(\cosh(r)-1)}{4(\cosh(r)-1)^2},
$$
which tends to $\pi^2/4$, as $r\to\infty$.  Together with Theorem \ref{thm:minkowski1}, this shows that if the optimal Minkowski's inequality in a Cartan-Hadamard $3$-space, with curvature $K\leq a\leq 0$, is of the form
$$
\M(\Gamma)\geq\sqrt{16\pi|\Gamma|-\lambda a|\Gamma|^2},
$$
then
$2\leq\lambda\leq \pi^2/4\simeq 2.47$. Hence the constant $2$ in \eqref{eq:minkowski1} is within \ $80\%$ of the largest value which it might possibly have.
\end{note}

\section{Smooth Strictly Convex Surfaces}\label{sec:smooth}
Here we prove Theorem \ref{thm:minkowski1} via harmonic mean curvature flow.
 A \emph{geometric flow}  of a hypersurface $\Gamma$ in a Riemannian $n$-manifold $M$ \cite{andrews-chow2020,giga2006,huisken-polden1999} is  a one parameter family of immersions $X\colon\Gamma\times[0,T)\to M$, $X_t(\cdot):=X(\cdot, t)$,  given by
\be\label{eq:hmcf}
X_t'(p)=-F_t(p)\nu_t(p),\quad\quad\quad X_0(p)=p,
\ee
where $(\cdot)':=\partial/\partial t(\cdot)$,
$\nu_t$ is a  normal vector field along $\Gamma_t:=X_t(\Gamma)$, and the \emph{speed function} $F_t$ depends on \emph{principal curvatures} or eigenvalues $\kappa_i^t$  of the second fundamental form $\mathrm{I\!I}_t:=\mathrm{I\!I}_{\Gamma_t}$. More precisely, $\nu_t(p)$ is the normal and $\kappa_i^t(p)$ are the principal curvatures of $\Gamma_t$ at the point $X_t(p)$.
 When $F_t$ is the harmonic mean of $\kappa_i^t$, i.e., 
$$
F_t=\left(\sum\frac{1}{\kappa_i^t}\right)^{-1},
$$
$X$ is called the \emph{harmonic mean curvature flow} of $\Gamma$. In particular when $n=3$,
$$
F_t=\frac{G_t}{H_t},
$$
 where $G_t:=\det(\mathrm{I\!I}_t)$ and $H_t:=\textup{trace}(\mathrm{I\!I}_t)$ are the \emph{Gauss-Kronecker curvature} and \emph{mean curvature} of $\Gamma_t$ respectively.
 Xu showed that \cite[Thm. 1.2]{xu2010, gulliver-xu2009} when $\Gamma$ is a smooth strictly convex hypersurface in a Cartan-Hadamard manifold $M$ and $F_t$ is the harmonic mean curvature, $X$ exists for $t\in [0,T)$, is $\C^\infty$, and $\Gamma_t$ are strictly convex hypersurfaces
converging to a point as $t\to T$. This is the only geometric flow known to preserve the convexity of a hypersurface in $M$ while contracting it to a point.

\begin{proof}[Proof of Theorem \ref{thm:minkowski1}]
Let $\Gamma_t$, $t\in[0,T)$, be the surfaces generated by the harmonic mean curvature flow of $\Gamma$, converging to a point $o$ in $M$. Set $\M_t:=\M(\Gamma_t)$, and 
$$
\phi(t):=\M_t^2-16\pi|\Gamma_t|+2a|\Gamma_t|^2.
$$
We need to show that $\phi(0)\geq 0$. To this end, we compute $\phi'$ as follows. 
It is well-known that \cite[Thm. 3.2(v)]{huisken-polden1999} for any geometric flow
$$
(H_t)' = 
\Delta_t F_t+\left(|\mathrm{I\!I}_t|^2+\textup{Ric}(\nu_t)\right)F_t,
$$
where $|\mathrm{I\!I}_t|:=\sqrt{\sum(\kappa_i^t)^2}$, $\Delta_t$ is the Laplace-Beltrami operator induced on $\Gamma$ by $X_t$, and $\textup{Ric}(\nu_t)$ is the Ricci curvature of $M$ at the point $X_t(p)$ in the direction of $\nu_t(p)$, i.e., the sum of sectional curvatures of $M$ with respect to a pair of orthogonal planes which contain $\nu_t(p)$. Let $d\mu_t$ be the area element induced on $\Gamma$ by $X_t$.
 By \cite[Lem. 7.4]{huisken-polden1999},
$$
(d\mu_t)'=-F_tH_td\mu_t=-G_t d\mu_t.
$$
Using the above formulas, we compute that
\begin{eqnarray}\label{eq:G'}\notag
\M'_t&=&\int_{\Gamma}\left((H_t)'d\mu_t+H_t(d\mu_t)'\right)\\ 
&=&
\int_{\Gamma}\Big(\Delta_t F_t+\big(|\mathrm{I\!I}_t|^2-(H_t)^2\big)F_t +\textup{Ric}(\nu_t)F_t\Big)d\mu_t\\ \notag
&\leq&
-2\int_{\Gamma}(G_t -a)\frac{G_t}{H_t}d\mu_t\\ \notag
&\leq&
-2\int_{\Gamma}\frac{(G_t)^2}{H_t}\,d\mu_t.
\end{eqnarray}
So, by Cauchy-Schwarz inequality, 
\begin{eqnarray}\label{eq:CS}
\M_t\M'_t
\leq
-2\M_t\int_{\Gamma}\frac{(G_t)^2}{H_t}d\mu_t
\leq
-2\G_t^2,
\end{eqnarray}
where $\G_t=\G(\Gamma_t):=\int_{\Gamma}G_t d\mu_t$ is the \emph{total Gauss-Kronecker curvature} of $\Gamma_t$.
Let $H$ be the function on $\Omega\setminus\{o\}$ given by $H(X_t(p)):=H_t(p)$. Also 
define $u$ on $\Omega\setminus\{o\}$ by  $u(X_t(p))=t$, which yields that $|\nabla u(X_t)|=1/F_t$. Then $H=\textup{div}(\nabla u/|\nabla u|)$, and Stokes' theorem together with the coarea formula yields that
$$
|\Gamma_t|-|\Gamma_{t+h}|=\int_{\Omega_t\setminus\Omega_{t+h}}H
=
\int_t^{t+h}\left(\int_{\Gamma}H_sF_s\,d\mu_s\right)ds
=
\int_t^{t+h}\G_sds,
$$
where $\Omega_t$ is the convex domain bounded by $\Gamma_t$.
Thus
$$
|\Gamma_t|'=-\G_t.
$$
It follows that
\be\label{eq:phi'}
\phi'(t) 
= 
2\M_t\M'_t
-16\pi |\Gamma_t|'+4a|\Gamma_t||\Gamma_t|'
\leq
-4\G_t\Big(\G_t-4\pi+a|\Gamma_t|\Big)
\leq
0,
\ee
where the last inequality is due to Gauss' equation and Gauss-Bonnet theorem. Indeed,
by Gauss' equation, for all $p\in\Gamma_t$, 
\be\label{eq:gauss}
G_t(p)=K_{\Gamma_t}(p)-K_M(T_p\Gamma_t),
\ee
where $K_{\Gamma_t}$ is the sectional curvature of $\Gamma_t$, and $K_M(T_p\Gamma_t)$ is the sectional curvature of $M$ with respect to the tangent plane $T_p \Gamma_t\subset T_pM$.   So, by Gauss-Bonnet theorem, 
\be\label{eq:GB}
\G_t=4\pi -\int_{p\in\Gamma_t} K_M(T_p\Gamma_t)\geq 4\pi-a|\Gamma_t|.
\ee
Hence $\phi'\leq 0$ as claimed. But since $\Gamma_t$ is convex and collapses to a point, 
 $|\Gamma_t|\to 0$, which yields that
 $$
 \lim_{t\to T}\phi(t)=\lim_{t\to T} \M_t^2\geq 0.
 $$ 
Thus $\phi(0)\geq 0$, which yields the desired inequality \eqref{eq:minkowski1}. 

If equality holds in \eqref{eq:minkowski1}, then $\phi(0)=0$ which yields $\phi(t)\equiv 0$, since $\phi(t)\geq 0$ and $\phi'(t)\leq 0$. Then $\phi'(t)\equiv 0$. So equalities hold in \eqref{eq:phi'}, which yields  $\M_t\M'_t=-2\G_t^2$. Consequently, the inequalities in \eqref{eq:CS} become equalities. This forces $G_t/H_t=\lambda(t)$, by the equality case in Cauchy-Schwarz inequality. So $\Gamma_t$ are parallel to $\Gamma$, which means that all points of $\Gamma$ have constant distance from $o$. Hence $\Gamma$ is a (geodesic) sphere. Finally, equalities in \eqref{eq:CS} force equalities in \eqref{eq:G'}. This forces $\textup{Ric}(\nu_t)\equiv 0$, which in turn yields that the sectional curvatures with respect to planes containing $\nu_t$ must vanish, since they are nonpositive. Consequently all sectional curvatures  of $M$ vanish in the (geodesic) ball bounded by $\Gamma$, by \cite[Lem. 5.4]{ghomi-spruck2022}, which completes the proof.
\end{proof}

\begin{note}\label{note:h-convex}
Andrews and Wei \cite{andrews-wei2018} showed that harmonic mean curvature flow preserves $h$-convexity of hypersurfaces in hyperbolic space. If this property holds  in any Cartan-Hadamard space,  then the proof of Theorem \ref{thm:minkowski1} may be refined to establish in that space the stronger inequality 
\be\label{eq:minkowski3}
\M(\Gamma)\geq \sqrt{16\pi|\Gamma|-\frac72\,a|\Gamma|^2},
\ee
 when $\Gamma$ is $h$-convex and $a<0$. To establish this claim, 
we rescale the metric of $M$ so that $a=-1$, for convenience. Then, similar to the proof of Theorem \ref{thm:minkowski1}, we set
$$
\phi(t):=\M_t^2-16\pi|\Gamma_t|-\frac72|\Gamma_t|^2,
$$
and compute that
\begin{eqnarray*}
\phi'(t)
=2\M_t\M_t'+\big(16\pi+7|\Gamma_t|\big)\G_t.
\end{eqnarray*}
Next recall that by \eqref{eq:G'}
\begin{eqnarray*}
\M_t\M_t'&\leq&-2\int H_t\int\frac{(G_t)^2 +G_t}{H_t}\,\\
&=& -2\int H_t\int\frac{(G_t+1/2)^2-1/4}{H_t}\\
&\leq& -2\left(\int \left(G_t+\frac{1}{2}\right)\right)^2+\frac{1}{2}\int H_t\int\frac{1}{H_t},
\end{eqnarray*}
where all integrals take place over $\Gamma$ with respect to $d\mu_t$.
If $\Gamma_t$ is $h$-convex, then its principal curvatures are $\geq 1$. Indeed the principal curvatures of horospheres in $M$ are bounded below by $1$, since principal curvatures of spheres of radius $\rho$ in $M$ are 
$\geq\coth(\rho)$ \cite[p. 184]{karcher1989}. It follows that  
\be\label{eq:HG}
2\leq H_t\leq 2G_t,
\ee
 which in turn yields
\be\label{eq:HH}
\int H_t\int\frac{1}{H_t}
\leq
\int 2G_t\int\frac12
=|\Gamma_t| \G_t.
\ee
Furthermore, since  by \eqref{eq:GB} $\G_t\geq 4\pi+|\Gamma|$,
$$
\left(\int \left(G_t+\frac{1}{2}\right)\right)^2=\G_t^2+\G_t|\Gamma_t|+\frac{1}{4}|\Gamma_t|^2
\geq \big(4\pi +2|\Gamma_t|\big)\G_t.
$$
So we have
\begin{eqnarray*}
\M_t\M_t'
\leq
-2\big(4\pi +2|\Gamma_t|\big)\G_t+\frac{1}{2}|\Gamma_t| \G_t
\leq 
-\left(8\pi +\frac{7}{2}|\Gamma_t|\right)\G_t,
\end{eqnarray*}
which in turn yields
$$
\phi'(t)
\leq-2\left(8\pi +\frac{7}{2}|\Gamma_t|\right)+\big(16\pi+7|\Gamma_t|\big)\G_t=0.
$$
Hence, since $\lim_{t\to T}\phi(t)\geq 0$, it follows that $\phi(0)\geq 0$, which establishes the desired inequality \eqref{eq:minkowski3}. Note that \eqref{eq:HG} was the only place where $h$-convexity was used in the above argument, which was for the sole purpose of establishing \eqref{eq:HH}. Thus \eqref{eq:minkowski3} holds whenever there exists a function $\lambda\colon [0,T)\to \R$ such that $\lambda(t)\leq H_t\leq \lambda(t) G_t$. 
\end{note}

\section{General Convex Surfaces}
Here we employ an approximation argument to extend \eqref{eq:minkowski1}, which we established for smooth strictly convex surfaces in the last section, to all convex surfaces in a Cartan-Hadamard $3$-manifold. This involves some basic facts about convex sets and their distance functions in Riemannian geometry which can be found in \cite[Sec. 2 \& 3]{ghomi-spruck2022} plus recent comparison results for total mean curvature obtained  in \cite{ghomi-spruck2023}. 

A subset of a Cartan-Hadamard manifold $M$ is \emph{convex} if it contains the geodesic segment connecting every pair of its points. A \emph{convex hypersurface} $\Gamma\subset M$ is the boundary of a compact convex set with interior points. Let $d_\Gamma\colon M\to\R$ be the \emph{distance function} of $\Gamma$, and $\Omega$ be the domain bounded by $\Gamma$. The \emph{signed} distance function of $\Gamma$ is defined by setting $\widehat d_\Gamma:=d_\Gamma$ on $M\setminus\Omega$ and $\widehat d_\Gamma:=-d_\Gamma$ on $\Omega$. The level sets 
$$
\Gamma_t:=\big(\widehat{d}_\Gamma\big)^{-1}(t)
$$ 
are called \emph{parallel hypersurfaces} of $\Gamma$. Unless noted otherwise, we assume that $t\geq 0$ and call $\Gamma_t$ the \emph{outer parallel} hypersurface, while $\Gamma_{-t}$ will be called the \emph{inner parallel} hypersurface of $\Gamma$. A fact which will be used frequently below is that  $\Gamma_t$ are $\C^{1,1}$ and convex for $t>0$ \cite[Sec. 2 \& 3]{ghomi-spruck2022}. In particular, for $t>0$, $\Gamma_t$ is twice differentiable almost everywhere and so its total mean curvature $\M(\Gamma_t)$ is well defined and positive. 

\begin{lemma}\label{lem:MGtToMG}
For any $\C^{1,1}$ convex hypersurface $\Gamma$ in a Cartan-Hadamard manifold, $t\mapsto\M(\Gamma_t)$ is a continuous nondecreasing function for $t\geq 0$.
\end{lemma}
\begin{proof}
Let $\Omega_t$ denote the domain bounded by  $\Gamma_t$.
By \cite[(12)]{ghomi-spruck2023}, for $0\leq t_1\leq t_2$,
\be\label{eq:continuity}
\M(\Gamma_{t_2})-\M(\Gamma_{t_1})=\int_{\Omega_{t_2}\setminus\Omega_{t_1}}\big(2\sigma_2(\kappa)-\textup{Ric}(\nabla \widehat{d}_\Gamma)\big),
\ee
where $\kappa=(\kappa_1,\dots,\kappa_{n-1})$ refers to the principal curvatures of parallel hypersurfaces of $\Gamma$ and $\sigma_2$ is the second symmetric elementary function. Since $\Gamma$ is convex, $\sigma_2(\kappa)\geq 0$, and by assumption the Ricci curvature of $M$ is nonpositive. Thus $t\mapsto\M(\Gamma_t)$ is nondecreasing. The above expression also yields the continuity of $t\mapsto\M(\Gamma_t)$, since the integrand depends only on $\Gamma$ and $M$. So the integral vanishes as $t_1\to t_2$, or $t_2\to t_1$.
\end{proof}

Now for any convex hypersurface $\Gamma$ in a Cartan-Hadamard manifold, which may not be $\C^{1,1}$, we set
$$
\M(\Gamma):=\lim _{t\to 0^+} \M(\Gamma_t).
$$
Since $\M(\Gamma_t)\geq 0$, and by Lemma \ref{lem:MGtToMG}, $\M(\Gamma_t)$ does not increase as $t\to 0^+$, the above limit exists. Furthermore, continuity of $t\mapsto\M(\Gamma_t)$ ensures that, in case $\Gamma$ is $\C^{1,1}$,  the above definition coincides with the regular definition of $\M(\Gamma)$ as the integral of mean curvature. Now that  $\M(\Gamma)$ is well-defined for all convex hypersurface, we may state the main result of this section:

\begin{theorem}\label{thm:minkowski2}
Minkowski's inequality \eqref{eq:minkowski1} holds for all convex surfaces $\Gamma$ in a Cartan-Hadamard $3$-manifold $M$ with curvature $K\leq a\leq 0$.
\end{theorem}

To establish this theorem we need  the following facts:

\begin{lemma}\label{lem:smoothing}
Smooth strictly convex hypersurfaces are dense in the space of $\C^k$ convex hypersurfaces of a Cartan-Hadamard manifold with respect to $\C^k$ topology, for $k\geq 0$.
\end{lemma}
\begin{proof}
Let $\Gamma$ be a convex hypersurface in a Cartan-Hadamard manifold $M$, and $u\colon M\to \R$ be the distance function from the domain $\Omega$ bounded by $\Gamma$. Let $x_0$ be a point in the interior of $\Omega$, and $\rho$ be the distance function from $x_0$. Then, for $\epsilon>0$, $u^\epsilon(x):=u(x)+\epsilon \rho^2(x)$ is a strictly convex function in the sense of Greene and Wu \cite{greene-wu1972}. Consequently, the Greene-Wu convolution $u^\epsilon_\lambda$ yields a family of smooth strictly convex functions converging to $u^\epsilon$ with respect to $\C^k$ norm over any compact set, as $\lambda\to 0$ \cite[Thm. 2 \& Lem. 3.3]{greene-wu1976}; see \cite[p. 21--22]{ghomi-spruck2022}. 
In particular, for any given integer $i>0$, we may choose $\epsilon$ and $\lambda$ so small that a level set $\Gamma^i$ of $u ^\epsilon_\lambda$ lies within a neighborhood of $\Gamma$ of radius $1/i$. Then $\Gamma_i$ converges to $\Gamma$ with respect to $\C^k$ topology, which completes the proof.
\end{proof}

We say that a set is \emph{nested inside} a convex hypersurface
$\Gamma$ provided that it lies in the convex domain bounded by $\Gamma$. The following monotonicity property is a quick consequence of an analogous result in \cite{ghomi-spruck2023} for $\C^{1,1}$ surfaces:

\begin{lemma}\label{lem:monotonicity}
Let $\gamma$, $\Gamma$ be a pair of of convex hypersurfaces in a Cartan-Hadamard manifold. Suppose that $\gamma$ is nested inside $\Gamma$. Then $\M(\gamma)\leq \M(\Gamma)$.
\end{lemma}
\begin{proof}
For every $t>0$, $\gamma_t$ and $\Gamma_t$ are $\C^{1,1}$ convex hypersurfaces, with $\gamma_t$ nested inside $\Gamma_t$. Thus $\M(\gamma_t)\leq \M(\Gamma_t)$ by \cite[Cor. 4.1]{ghomi-spruck2023}. Letting $t\to 0$ completes the proof.
\end{proof}

The next observation follows from the  fact that the nearest point projection into a convex set is distance nonincreasing in Cartan-Hadamard manifolds \cite[Prop. 2.4(4)]{bridson-haefliger1999}.

\begin{lemma}\label{lem:monotone-area}
Let $\gamma$, $\Gamma$ be a pair of convex hypersurfaces in a Cartan-Hadamard manifold, with $\gamma$ nested inside $\Gamma$. Then $|\gamma|\leq |\Gamma|$, with equality only if $\gamma=\Gamma$.
\end{lemma}

Now were are ready to establish the main result of this section:

\begin{proof}[Proof of Theorem \ref{thm:minkowski2}]
By Lemma \ref{lem:smoothing}, there exists a family $\Gamma^i\subset M$ of smooth strictly convex hypersurfaces which converge to $\Gamma$ with respect to $\C^0$ topology. After replacing each $\Gamma^i$ by an outer parallel hypersurface, we may assume that $\Gamma^i\subset M\setminus\Omega$, where $\Omega$ is the domain bounded by $\Gamma$.  By Theorem \ref{thm:minkowski1}, $\M(\Gamma^i)$ satisfy \eqref{eq:minkowski1}. Thus it suffices to check that $|\Gamma^i|\to |\Gamma|$ and  $\M(\Gamma^i)\to\M(\Gamma)$. For every $\epsilon>0$, there exists  an integer $N$ such that $\Gamma^i$ lies in the region bounded by $\Gamma$ and $\Gamma_{\epsilon}$ for $i\geq N$. Thus
$$
|\Gamma|\leq |\Gamma^i|\leq|\Gamma_\epsilon|,
\quad\text{and}\quad
\M(\Gamma)\leq \M(\Gamma^i)\leq \M(\Gamma_{\epsilon}),
$$
by Lemmas \ref{lem:monotone-area} and \ref{lem:monotonicity}. 
As $\epsilon\to 0$, $|\Gamma_\epsilon|\to |\Gamma|$, and by Lemma \ref{lem:MGtToMG},
$\M(\Gamma_\epsilon)\to \M(\Gamma)$ as well, which completes the proof.
\end{proof}

\section{Isoperimetric Inequality}\label{sec:isoperimetric}
Here we prove Theorem \ref{thm:isoperimetric}, using the generalized Minkowski's inequality \eqref{eq:minkowski1} derived in the last section, and a Steiner type formula which we will establish below. To this end we need to define the total Gauss-Kronecker curvature of a general convex hypersurface $\Gamma$ in a Cartan-Hadamard manifold. Similar to our treatment for total mean curvature in the last section, we set
\be\label{eq:defG}
\G(\Gamma):=\lim_{t\to 0^+}\G(\Gamma_t),
\ee
where recall that $\Gamma_t$ denote the outer parallel hypersurfaces of $\Gamma$.
By \cite[Cor. 4.4]{ghomi-spruck2023}, $\G(\Gamma_t)$ does not increase as $t\to 0^+$. Thus, since $\G(\Gamma_t)\geq 0$, the above limit exists. Let us also record that:

\begin{lemma}\label{lem:GK-continuous}
The total Gauss-Kronecker curvature $\G(\Gamma)$ is continuous in the space of $\C^{1,1}$ convex surfaces $\Gamma$ in a Cartan-Hadamard $3$-manifold $M$ with respect to $\C^1$ topology.
\end{lemma}
\begin{proof}
Let $\Gamma^i$ be a sequence of $\C^{1,1}$ convex surfaces in $M$ converging to $\Gamma$ with respect to $\C^1$ topology. Then Gauss' equation \eqref{eq:gauss} together with Gauss-Bonnet theorem  yields
$$
\G(\Gamma^i)=4\pi -\int_{p\in\Gamma^i} K_M(T_p\Gamma^i)\quad\longrightarrow\quad 4\pi -\int_{p\in\Gamma} K_M(T_p\Gamma)=\G(\Gamma),
$$
as desired.
\end{proof}

Now we can establish the following Steiner type formula for general convex surfaces, using tube formulas of Gray \cite{gray2004} together with an approximation argument.

\begin{lemma}\label{lem:steiner}
Let $\Gamma$ be a convex surface in a Cartan-Hadamard $3$-manifold. Then, for any $t\geq 0$,
\be\label{eq:steiner}
|\Gamma_t| \geq |\Gamma|+\M(\Gamma)t+\G(\Gamma)t^2.
\ee
\end{lemma}
\begin{proof}
If $\Gamma$ is smooth, then \eqref{eq:steiner} holds by  Steiner's formula in spaces of nonpositive curvature \cite[Thm. 10.31(ii)]{gray2004}. The general case follows by approximation. If for every $\epsilon>0$ we can show 
$$
|\Gamma_{\epsilon+t}| \geq |\Gamma_\epsilon|+\M(\Gamma_\epsilon)t+\G(\Gamma_\epsilon)t^2,
$$
then \eqref{eq:steiner} follows by letting $\epsilon\to 0$. Hence we may assume that $\Gamma$ is $\C^{1,1}$, after replacing $\Gamma$ with $\Gamma_\epsilon$.
Then, by Lemma \ref{lem:smoothing}, there exists a family of smooth convex surfaces $\Gamma^i\subset M$ such that $\Gamma^i\to\Gamma$ with respect to $\C^1$ topology. As described in the proof of Theorem \ref{thm:minkowski2},  we may assume that $\Gamma^i$ lie outside the domain bounded by $\Gamma$, which yields  $\M(\Gamma^i)\to \M(\Gamma)$ via Lemma \ref{lem:monotonicity}. Furthermore, $\G(\Gamma^i)\to\G(\Gamma)$ as well, by Lemma \ref{lem:GK-continuous}. Finally  $|\Gamma^i|\to|\Gamma|$ and $|(\Gamma^i)_t|\to |\Gamma_t|$ by Lemma \ref{lem:monotone-area}, as shown in the proof of Theorem \ref{thm:minkowski2}. Thus, as $\Gamma^i$ satisfy \eqref{eq:steiner}, $\Gamma$ does as well.
\end{proof}

The next fact is well-known when $\Gamma$ is smooth. For the sake of completeness, we quickly check that it holds under minimal regularity:

\begin{lemma}\label{lem:var-of-area}
Let $\Gamma$ be a closed oriented $\C^{1,1}$  hypersurface embedded in a Riemannian manifold $M$, and $\Gamma_t$ be the parallel hypersurfaces of $\Gamma$ for $-\epsilon<t<\epsilon$. Then
$$
\M(\Gamma)=\frac{d}{dt}\Big|_{t=0}|\Gamma_t|.
$$
\end{lemma}
\begin{proof}
Since $\Gamma$ is $\C^{1,1}$, the signed distance function $\widehat{d}_\Gamma$ of $\Gamma$ is $\C^{1,1}$ on an open neighborhood $U$ of $\Gamma$ in $M$ \cite[Lem. 2.6]{ghomi-spruck2022}. Thus  $H=\textup{div}(\nabla \widehat{d}_\Gamma)$ almost everywhere on $U$, where $H$ is the mean curvature of parallel hypersurfaces of $\Gamma$ in $U$. Consequently, by Stokes' theorem and the coarea formula, for $t\geq 0$,
$$
|\Gamma_t|-|\Gamma|=\int_{\Lambda_t}\textup{div}(\nabla \widehat{d}_\Gamma)=\int_{\Lambda_t}H=\int_0^t\M(\Gamma_s)ds,
$$
where $\Lambda_t$ is the domain bounded  between $\Gamma$ and $\Gamma_t$. Furthermore, by \eqref{eq:continuity}, $\M(\Gamma_s)$ is continuous for $0\leq s\leq t$ (formula \eqref{eq:continuity} holds for all pairs of parallel closed $\C^{1,1}$ hyeprsurfaces in Riemannian manifolds \cite[Thm. 3.1]{ghomi-spruck2023}). Thus, 
by the mean value theorem for integrals, the right derivative of $|\Gamma_t|$ at $t=0$ is equal to $\M(\Gamma)$. 
Similarly, the left derivative at $t=0$ is equal to $\M(\Gamma)$, which completes the proof.
\end{proof}

The last observation we need follows quickly from Lemma \ref{lem:monotone-area} and the fact that the exponential map is distance nonreducing in Cartan-Hadamard manifolds:

\begin{lemma}\label{lem:balls}
Let $\Gamma$ be a convex hypersurface in a Cartan-Hadamard $n$-manifold bounding a domain $\Omega$ with inradius $r$, and $S$ be a sphere in $\R^n$ with radius $R$. Suppose that $|\Gamma|=|S|$.  Then $r\leq R$, with equality only if $\Omega$ is isometric to a ball in $\R^n$.
\end{lemma}

Now we are ready to establish the main result of this section:

\begin{proof}[Proof of Theorem \ref{thm:isoperimetric}]
Let  $S\subset\R^3$ be a sphere with $|S|=|\Gamma|$ and radius $R$. By Lemma \ref{lem:balls}, 
$R\geq r:=\textup{inrad}(\Omega)$. Let $\Gamma_{-t}$, $S_{-t}$ denote respectively the inner parallel surfaces of $\Gamma$,  $S$ at distance $t\in (0,r)$, as defined in the last section.   It is enough to show that $|\Gamma_{-t}|\leq |S_{-t}|$, for then by the coarea formula
\be\label{eq:Omega-A}
|\Omega|=\int_0^r|\Gamma_{-t}|dt \leq \int_0^r|S_{-t}|dt=|\Lambda|,
\ee
where $\Lambda$ is the annular region between $S$ and $S_{-r}$. The application of the coarea formula here is warranted since 
the distance function $\widehat{d}_\Gamma$ is Lipschitz and $|\nabla \widehat{d}_\Gamma|=1$ almost everywhere \cite[Sec. 2]{ghomi-spruck2022}. Furthermore $|\Lambda|$ is the desired upper bound, since $R=\sqrt{|\Gamma|/(4\pi)}$.   
Now suppose, towards a contradiction, that  
$$
|\Gamma_{-t_0}|> |S_{-t_0}|
$$
 for some $t_0\in(0,r)$.  Since  $\Gamma$ is $d$-convex,  $\Gamma_{- t}$ are convex.
So, by Lemmas  \ref{lem:monotone-area} and \ref{lem:steiner},
$$
|\Gamma| \geq |(\Gamma_{-t_0})_{t_0}| \geq |\Gamma_{-t_0}|+\M(\Gamma_{-t_0})t_0+\G(\Gamma_{-t_0})t_0^2.
$$
There exists $s_0> 0$ such that $|\Gamma_{-t_0}|=|S_{-t_0+s_0}|$. By Theorem \ref{thm:minkowski2},
$$
\M(\Gamma_{-t_0})\geq \M(S_{-t_0+s_0}) > \M(S_{-t_0}).
$$
Furthermore, recall that by Gauss' equation and Gauss-Bonnet theorem \eqref{eq:GB},
$
\G((\Gamma_{-t_0})_s)\geq 4\pi= \G(S_{-t_0}),
$
for all $s>0$. So 
$$
\G(\Gamma_{-t_0})\geq \G(S_{-t_0}),
$$
by definition \eqref{eq:defG}.
 Hence we obtain
$$
|\Gamma|> |S_{-t_0}|+\M(S_{-t_0})t_0+\G(S_{-t_0})t_0^2 = |(S_{-t_0})_{t_0}|=|S|=|\Gamma|,
$$
which is the desired contradiction. 
Finally, suppose that equality holds in \eqref{eq:bonnesen}. Then equality holds in \eqref{eq:Omega-A}. Consequently $|\Gamma_{-t}|= |S_{-t}|$, since we just showed that $|\Gamma_{-t}|\leq |S_{-t}|$. It follows, via Lemma \ref{lem:var-of-area}, that 
$$
\M(\Gamma)=\frac{d}{dt}|\Gamma_{t}|\Big|_{t=0}=-\frac{d}{dt}|\Gamma_{-t}|\Big|_{t=0}=-\frac{d}{dt}|S_{-t}|\Big|_{t=0}=\M(S).
$$
So, by Theorem \ref{thm:minkowski1}, $\Gamma$ must bound a Euclidean ball.
\end{proof}
\section{Higher Dimensions}\label{sec:high-dim}
Let $\Gamma$ be a convex hypersurface in a Cartan-Hadamard $n$-manifold $M$ bounding a domain $\Omega$, and $S$ be a sphere  in $\R^n$ with $|S|=|\Gamma|$ bounding a ball $B$.
The analogue of Minkowski inequality \eqref{eq:minkowski0} in  $M$ is that 
\be\label{eq:gen-minkowski}
\M(\Gamma)\geq \M(S),
\ee
with equality only if $\Omega$ is isometric to $B$. The analogue of the  isoperimetric inequality \eqref{eq:bonnesen}  is that, if   $\Gamma$ is $d$-convex, $r:=\textup{inrad}(\Omega)$, and $R$ is the radius of $B$, then
\be\label{eq:gen-bonnesen}
|\Omega|\leq |B\setminus B_{R-r}|,
\ee
where  $B_\rho$ stands for a ball of radius $\rho$ in $\R^n$ with the same center as $B$, and equality holds only if $\Omega$ is isometric to $B$. By Lemma \ref{lem:balls}, $r\leq R$, thus $B_{R-r}$ is well-defined.

\begin{theorem}\label{thm:higher}
Let $M$ be a Cartan-Hadamard $n$-manifold. Suppose that  the Minkowski type inequality \eqref{eq:gen-minkowski} holds for all $\C^{1,1}$ convex hypersurfaces $\Gamma\subset M$ with equality only if the domain $\Omega$ bounded by $\Gamma$ is isometric to a ball in $\R^n$. Then the isoperimetric inequality \eqref{eq:gen-bonnesen} also holds in $M$ for all domains $\Omega$ with $\C^{1,1}$ $d$-convex boundary $\Gamma$, and equality holds only if $\Omega$ is isometric to a ball in $\R^n$. 
\end{theorem}

The proof of the above theorem uses the notion of reach in the sense of Federer \cite{federer1959,thale2008}, see \cite[Sec. 2]{ghomi-spruck2022}. The \emph{reach} of a convex hypersurface $\Gamma\subset M$, bounding a domain $\Omega$, is the supremum value of $\rho$ such that  through each point of $\Gamma$ there passes a  ball of radius $\rho$ contained in $\Omega$. It is well-known that $\textup{reach}(\Gamma)>0$ if and only if $\Gamma$ is $\C^{1,1}$ \cite[Lem. 2.6]{ghomi-spruck2022}. Lemma \ref{lem:monotone-area} quickly yields:

\begin{lemma}\label{lem:reach}
Let $\Gamma$ be a $d$-convex hypersurface in a Cartan-Hadamard manifold bounding a domain $\Omega$, and $t\in[0,\textup{inrad}(\Omega))$. Then $|(\Gamma_{-t})_t|\leq|\Gamma|$ with equality if and only if $t\leq \textup{reach}(\Gamma)$.
\end{lemma}

Now we are ready to establish the main result of this section:

\begin{proof}[Proof of Theorem \ref{thm:higher}]
 If $\Omega$ is isometric to $B$, then equality holds in \eqref{eq:gen-bonnesen} and there is nothing left to prove. 
 So suppose  that $\Omega$ is not isometric to $B$.
Then, by the rigidity assumption for Minkowski's inequality \eqref{eq:gen-minkowski}, $\M(\Gamma)>\M(S)$. Consequently, since $\Gamma$ is $\C^{1,1}$, there exists $\epsilon>0$ such that
\be\label{eq:Gamma-t-S-t}
|\Gamma_{-t}|<|S_{-t}|, \quad\quad \text{for}\quad\quad t\in(0,\epsilon],
\ee
by Lemma \ref{lem:var-of-area}.
Furthermore note that, by Lemma \ref{lem:balls}, 
\be\label{eq:r-R}
r<R.
\ee 
So $S_{-t}$ is well-defined for $t\in (0,r)$. 
If $|\Gamma_{-t}|\leq |S_{-t}|$ for all $t\in [\epsilon,r)$, then by the coarea formula
$$
|\Omega|=\int_0^r|\Gamma_{-t}|\,dt<\int_0^r|S_{-t}|\,dt=|B\setminus B_{R-r}|,
$$
and we are done.
So suppose, towards a contradiction, that
$
|\Gamma_{-t_0}|> |S_{-t_0}|,
$
for some $t_0\in [\epsilon,r)$. Let
$$
\ol{s}:=\sup\Big\{\,s\leq t_0\;\,\Big|\;\;(\Gamma_{-t_0})_t|\geq|S_{-t_0+t}| \;\text{   for all   }\; t\in[0,s]\,\Big\}.
$$
Then $\ol s>0$, and 
\be\label{eq:Gamma-t0-s}
|(\Gamma_{-t_0})_{\ol s}|=|S_{-t_0+\ol{s}}|.
\ee
Note that $(\Gamma_{-t_0})_{\ol s}$ cannot bound a Euclidean ball, for otherwise the radius of $(\Gamma_{-t_0})_{\ol s}$ would  be equal  to that  of $S_{-t_0+\ol{s}}$, i.e., $r-t_0+s=R-t_0+s$, or $r=R$, which would violate \eqref{eq:r-R}.
Consequently, by the rigidity assumption for \eqref{eq:gen-minkowski},
\be\label{eq:Gamma-t0-s-2}
\M(\Gamma_{-t_0})_{\ol s}>\M(S_{-t_0+\ol{s}}).
\ee
But since $\ol s>0$, $(\Gamma_{-t_0})_{\ol s}$ is $\C^{1,1}$. 
So  \eqref{eq:Gamma-t0-s} and \eqref{eq:Gamma-t0-s-2} yield that
$|(\Gamma_{-t_0})_{{\ol s}+\delta}|>|S_{-t_0+\ol{s}+\delta}|$ for $\delta$ small, via Lemma \ref{lem:var-of-area}.
Hence
$$
\ol s=t_0,
$$
by the definition of $\ol s$. There are now two possibilities: either  $t_0>\textup{reach}(\Gamma)$, or $t_0\leq\textup{reach}(\Gamma)$. If $t_0>\textup{reach}(\Gamma)$, then by Lemma \ref{lem:reach},
$$
|(\Gamma_{t_0})_{\ol s}|<|\Gamma_{-t_0+\ol s}|=|\Gamma|=|S|=|S_{-t_0+\ol s}|,
$$
which is not possible by \eqref{eq:Gamma-t0-s}. If, on the other hand, $t_0\leq\textup{reach}(\Gamma)$, then, again by Lemma \ref{lem:reach} and the definition of $\ol s$,
$$
|\Gamma_{-t_0+t}|=|(\Gamma_{-t_0})_t|\geq |S_{-t_0+t}|,
$$
for all $t\in [0,t_0]$, which violates \eqref{eq:Gamma-t-S-t}, for $t$ close to $t_0$. So we arrive at the desired contradiction.
\end{proof}

\addtocontents{toc}{\protect\setcounter{tocdepth}{0}}
\section*{Acknowledgments}
We thank Igor Belegradek, Ramon van Handel, Yingxiang Hu, Haizhong Li, Emanuel Milman, and Changwei Xiong for useful and stimulating communications.

\addtocontents{toc}{\protect\setcounter{tocdepth}{1}}
\bibliography{references}

\begin{bibdiv}
\begin{biblist}

\bib{andrews-chow2020}{book}{
      author={Andrews, Ben},
      author={Chow, Bennett},
      author={Guenther, Christine},
      author={Langford, Mat},
       title={Extrinsic geometric flows},
      series={Graduate Studies in Mathematics},
   publisher={American Mathematical Society, Providence, RI},
        date={[2020] \copyright 2020},
      volume={206},
        ISBN={978-1-4704-5596-5},
      review={\MR{4249616}},
}

\bib{andrews-hu-li-2020}{article}{
      author={Andrews, Ben},
      author={Hu, Yingxiang},
      author={Li, Haizhong},
       title={Harmonic mean curvature flow and geometric inequalities},
        date={2020},
        ISSN={0001-8708},
     journal={Adv. Math.},
      volume={375},
       pages={107393, 28},
  url={https://doi-org.prx.library.gatech.edu/10.1016/j.aim.2020.107393},
      review={\MR{4170217}},
}

\bib{andrews-wei2018}{article}{
      author={Andrews, Ben},
      author={Wei, Yong},
       title={Quermassintegral preserving curvature flow in hyperbolic space},
        date={2018},
        ISSN={1016-443X},
     journal={Geom. Funct. Anal.},
      volume={28},
      number={5},
       pages={1183\ndash 1208},
  url={https://doi-org.prx.library.gatech.edu/10.1007/s00039-018-0456-9},
      review={\MR{3856791}},
}

\bib{brendle-wang}{article}{
      author={Brendle, Simon},
      author={Hung, Pei-Ken},
      author={Wang, Mu-Tao},
       title={A {M}inkowski inequality for hypersurfaces in the anti--de
  {S}itter--{S}chwarzschild manifold},
        date={2016},
        ISSN={0010-3640},
     journal={Comm. Pure Appl. Math.},
      volume={69},
      number={1},
       pages={124\ndash 144},
         url={https://doi.org/10.1002/cpa.21556},
      review={\MR{3433631}},
}

\bib{bridson-haefliger1999}{book}{
      author={Bridson, Martin~R.},
      author={Haefliger, Andr\'{e}},
       title={Metric spaces of non-positive curvature},
      series={Grundlehren der Mathematischen Wissenschaften [Fundamental
  Principles of Mathematical Sciences]},
   publisher={Springer-Verlag, Berlin},
        date={1999},
      volume={319},
        ISBN={3-540-64324-9},
         url={https://doi.org/10.1007/978-3-662-12494-9},
      review={\MR{1744486}},
}

\bib{chang-wang2013}{article}{
      author={Chang, Sun-Yung~Alice},
      author={Wang, Yi},
       title={Inequalities for quermassintegrals on {$k$}-convex domains},
        date={2013},
        ISSN={0001-8708},
     journal={Adv. Math.},
      volume={248},
       pages={335\ndash 377},
  url={https://doi-org.prx.library.gatech.edu/10.1016/j.aim.2013.08.006},
      review={\MR{3107515}},
}

\bib{choe-ritore2007}{article}{
      author={Choe, Jaigyoung},
      author={Ritor\'{e}, Manuel},
       title={The relative isoperimetric inequality in {C}artan-{H}adamard
  3-manifolds},
        date={2007},
        ISSN={0075-4102},
     journal={J. Reine Angew. Math.},
      volume={605},
       pages={179\ndash 191},
         url={https://doi-org.prx.library.gatech.edu/10.1515/CRELLE.2007.031},
      review={\MR{2338131}},
}

\bib{croke1984}{article}{
      author={Croke, Christopher~B.},
       title={A sharp four-dimensional isoperimetric inequality},
        date={1984},
        ISSN={0010-2571},
     journal={Comment. Math. Helv.},
      volume={59},
      number={2},
       pages={187\ndash 192},
         url={https://doi.org/10.1007/BF02566344},
      review={\MR{749103}},
}

\bib{dalphin2016}{article}{
      author={Dalphin, J.},
      author={Henrot, A.},
      author={Masnou, S.},
      author={Takahashi, T.},
       title={On the minimization of total mean curvature},
        date={2016},
        ISSN={1050-6926},
     journal={J. Geom. Anal.},
      volume={26},
      number={4},
       pages={2729\ndash 2750},
         url={https://doi.org/10.1007/s12220-015-9646-y},
      review={\MR{3544938}},
}

\bib{lima-girao2016}{article}{
      author={de~Lima, Levi~Lopes},
      author={Gir\~{a}o, Frederico},
       title={An {A}lexandrov-{F}enchel-type inequality in hyperbolic space
  with an application to a {P}enrose inequality},
        date={2016},
        ISSN={1424-0637},
     journal={Ann. Henri Poincar\'{e}},
      volume={17},
      number={4},
       pages={979\ndash 1002},
  url={https://doi-org.prx.library.gatech.edu/10.1007/s00023-015-0414-0},
      review={\MR{3472630}},
}

\bib{dekster1981}{article}{
      author={Dekster, B.~V.},
       title={Monotonicity of integral {G}auss curvature},
        date={1981},
        ISSN={0022-040X},
     journal={J. Differential Geom.},
      volume={16},
      number={2},
       pages={281\ndash 291},
  url={http://projecteuclid.org.prx.library.gatech.edu/euclid.jdg/1214436104},
      review={\MR{638793}},
}

\bib{federer1959}{article}{
      author={Federer, Herbert},
       title={Curvature measures},
        date={1959},
        ISSN={0002-9947},
     journal={Trans. Amer. Math. Soc.},
      volume={93},
       pages={418\ndash 491},
      review={\MR{MR0110078 (22 \#961)}},
}

\bib{gallego-solanes}{article}{
      author={Gallego, Eduardo},
      author={Solanes, Gil},
       title={Integral geometry and geometric inequalities in hyperbolic
  space},
        date={2005},
        ISSN={0926-2245},
     journal={Differential Geom. Appl.},
      volume={22},
      number={3},
       pages={315\ndash 325},
         url={https://doi.org/10.1016/j.difgeo.2005.01.006},
      review={\MR{2166125}},
}

\bib{ge-wang-wu2014}{article}{
      author={Ge, Yuxin},
      author={Wang, Guofang},
      author={Wu, Jie},
       title={Hyperbolic {A}lexandrov-{F}enchel quermassintegral inequalities
  {II}},
        date={2014},
        ISSN={0022-040X},
     journal={J. Differential Geom.},
      volume={98},
      number={2},
       pages={237\ndash 260},
  url={http://projecteuclid.org.prx.library.gatech.edu/euclid.jdg/1406552250},
      review={\MR{3263518}},
}

\bib{ghomi-spruck2022}{article}{
      author={Ghomi, Mohammad},
      author={Spruck, Joel},
       title={Total curvature and the isoperimetric inequality in
  {C}artan-{H}adamard manifolds},
        date={2022},
        ISSN={1050-6926},
     journal={J. Geom. Anal.},
      volume={32},
      number={2},
       pages={Paper No. 50, 54},
         url={https://doi.org/10.1007/s12220-021-00801-2},
      review={\MR{4358702}},
}

\bib{ghomi-spruck2023}{article}{
      author={Ghomi, Mohammad},
      author={Spruck, Joel},
       title={Total mean curvatures of {R}iemannian hypersurfaces},
        date={2023},
        ISSN={1536-1365},
     journal={Adv. Nonlinear Stud.},
      volume={23},
      number={1},
       pages={Paper No. 20220029, 10},
         url={https://doi.org/10.1515/ans-2022-0029},
      review={\MR{4530499}},
}

\bib{gibbons1997}{article}{
      author={Gibbons, G.~W.},
       title={Collapsing shells and the isoperimetric inequality for black
  holes},
        date={1997},
        ISSN={0264-9381},
     journal={Classical Quantum Gravity},
      volume={14},
      number={10},
       pages={2905\ndash 2915},
         url={https://doi.org/10.1088/0264-9381/14/10/016},
      review={\MR{1476553}},
}

\bib{giga2006}{book}{
      author={Giga, Yoshikazu},
       title={Surface evolution equations},
      series={Monographs in Mathematics},
   publisher={Birkh\"{a}user Verlag, Basel},
        date={2006},
      volume={99},
        ISBN={978-3-7643-2430-8; 3-7643-2430-9},
        note={A level set approach},
      review={\MR{2238463}},
}

\bib{gray2004}{book}{
      author={Gray, Alfred},
       title={Tubes},
     edition={Second},
      series={Progress in Mathematics},
   publisher={Birkh\"{a}user Verlag, Basel},
        date={2004},
      volume={221},
        ISBN={3-7643-6907-8},
  url={https://doi-org.prx.library.gatech.edu/10.1007/978-3-0348-7966-8},
        note={With a preface by Vicente Miquel},
      review={\MR{2024928}},
}

\bib{greene-wu1972}{article}{
      author={Greene, R.~E.},
      author={Wu, H.},
       title={On the subharmonicity and plurisubharmonicity of geodesically
  convex functions},
        date={1972/73},
        ISSN={0022-2518},
     journal={Indiana Univ. Math. J.},
      volume={22},
       pages={641\ndash 653},
  url={https://doi-org.prx.library.gatech.edu/10.1512/iumj.1973.22.22052},
      review={\MR{0422686}},
}

\bib{greene-wu1976}{article}{
      author={Greene, R.~E.},
      author={Wu, H.},
       title={{$C^{\infty }$} convex functions and manifolds of positive
  curvature},
        date={1976},
        ISSN={0001-5962},
     journal={Acta Math.},
      volume={137},
      number={3-4},
       pages={209\ndash 245},
         url={https://doi-org.prx.library.gatech.edu/10.1007/BF02392418},
      review={\MR{0458336}},
}

\bib{guan-li2009}{article}{
      author={Guan, Pengfei},
      author={Li, Junfang},
       title={The quermassintegral inequalities for {$k$}-convex starshaped
  domains},
        date={2009},
        ISSN={0001-8708},
     journal={Adv. Math.},
      volume={221},
      number={5},
       pages={1725\ndash 1732},
  url={https://doi-org.prx.library.gatech.edu/10.1016/j.aim.2009.03.005},
      review={\MR{2522433}},
}

\bib{gulliver-xu2009}{article}{
      author={Gulliver, Robert},
      author={Xu, Guoyi},
       title={Examples of hypersurfaces flowing by curvature in a {R}iemannian
  manifold},
        date={2009},
        ISSN={1019-8385},
     journal={Comm. Anal. Geom.},
      volume={17},
      number={4},
       pages={701\ndash 719},
  url={https://doi-org.prx.library.gatech.edu/10.4310/CAG.2009.v17.n4.a6},
      review={\MR{2601350}},
}

\bib{hoisington-mcgrath2022}{article}{
      author={Hoisington, Joseph~Ansel},
      author={McGrath, Peter},
       title={Symmetry and isoperimetry for {R}iemannian surfaces},
        date={2022},
        ISSN={0944-2669},
     journal={Calc. Var. Partial Differential Equations},
      volume={61},
      number={1},
       pages={Paper No. 6, 13},
  url={https://doi-org.prx.library.gatech.edu/10.1007/s00526-021-02117-z},
      review={\MR{4342996}},
}

\bib{huisken-polden1999}{incollection}{
      author={Huisken, Gerhard},
      author={Polden, Alexander},
       title={Geometric evolution equations for hypersurfaces},
        date={1999},
   booktitle={Calculus of variations and geometric evolution problems
  ({C}etraro, 1996)},
      series={Lecture Notes in Math.},
      volume={1713},
   publisher={Springer, Berlin},
       pages={45\ndash 84},
         url={https://doi-org.prx.library.gatech.edu/10.1007/BFb0092669},
      review={\MR{1731639}},
}

\bib{karcher1989}{incollection}{
      author={Karcher, Hermann},
       title={Riemannian comparison constructions},
        date={1989},
   booktitle={Global differential geometry},
      series={MAA Stud. Math.},
      volume={27},
   publisher={Math. Assoc. America, Washington, DC},
       pages={170\ndash 222},
      review={\MR{1013810}},
}

\bib{kleiner1992}{article}{
      author={Kleiner, Bruce},
       title={An isoperimetric comparison theorem},
        date={1992},
        ISSN={0020-9910},
     journal={Invent. Math.},
      volume={108},
      number={1},
       pages={37\ndash 47},
         url={https://doi-org.prx.library.gatech.edu/10.1007/BF02100598},
      review={\MR{1156385}},
}

\bib{kloeckner-kuperberg2019}{article}{
      author={Kloeckner, Beno\^{\i}t~R.},
      author={Kuperberg, Greg},
       title={The {C}artan-{H}adamard conjecture and the {L}ittle {P}rince},
        date={2019},
        ISSN={0213-2230},
     journal={Rev. Mat. Iberoam.},
      volume={35},
      number={4},
       pages={1195\ndash 1258},
         url={https://doi-org.prx.library.gatech.edu/10.4171/rmi/1082},
      review={\MR{3988083}},
}

\bib{minkowski1903}{article}{
      author={Minkowski, Hermann},
       title={Volumen und {O}berfl\"{a}che},
        date={1903},
        ISSN={0025-5831},
     journal={Math. Ann.},
      volume={57},
      number={4},
       pages={447\ndash 495},
         url={https://doi.org/10.1007/BF01445180},
      review={\MR{1511220}},
}

\bib{montiel-ros2009}{book}{
      author={Montiel, Sebasti\'{a}n},
      author={Ros, Antonio},
       title={Curves and surfaces},
     edition={Second},
      series={Graduate Studies in Mathematics},
   publisher={American Mathematical Society, Providence, RI; Real Sociedad
  Matem\'{a}tica Espa\~{n}ola, Madrid},
        date={2009},
      volume={69},
        ISBN={978-0-8218-4763-3},
         url={https://doi.org/10.1090/gsm/069},
        note={Translated from the 1998 Spanish original by Montiel and edited
  by Donald Babbitt},
      review={\MR{2522595}},
}

\bib{natario2015}{article}{
      author={Nat\'{a}rio, Jos\'{e}},
       title={A {M}inkowski-type inequality for convex surfaces in the
  hyperbolic 3-space},
        date={2015},
        ISSN={0926-2245},
     journal={Differential Geom. Appl.},
      volume={41},
       pages={102\ndash 109},
         url={https://doi.org/10.1016/j.difgeo.2015.05.002},
      review={\MR{3353742}},
}

\bib{osserman1978}{article}{
      author={Osserman, Robert},
       title={The isoperimetric inequality},
        date={1978},
        ISSN={0002-9904},
     journal={Bull. Amer. Math. Soc.},
      volume={84},
      number={6},
       pages={1182\ndash 1238},
  url={https://doi-org.prx.library.gatech.edu/10.1090/S0002-9904-1978-14553-4},
      review={\MR{500557}},
}

\bib{osserman1979}{article}{
      author={Osserman, Robert},
       title={Bonnesen-style isoperimetric inequalities},
        date={1979},
        ISSN={0002-9890},
     journal={Amer. Math. Monthly},
      volume={86},
      number={1},
       pages={1\ndash 29},
         url={https://doi-org.prx.library.gatech.edu/10.2307/2320297},
      review={\MR{519520}},
}

\bib{ritore2010}{book}{
      author={Ritor\'{e}, Manuel},
      author={Sinestrari, Carlo},
       title={Mean curvature flow and isoperimetric inequalities},
      series={Advanced Courses in Mathematics. CRM Barcelona},
   publisher={Birkh\"{a}user Verlag, Basel},
        date={2010},
        ISBN={978-3-0346-0212-9},
  url={https://doi-org.prx.library.gatech.edu/10.1007/978-3-0346-0213-6},
        note={Edited by Vicente Miquel and Joan Porti},
      review={\MR{2590630}},
}

\bib{santalo1963}{article}{
      author={Santal\'{o}, L.~A.},
       title={A relation between the mean curvatures of parallel convex bodies
  in spaces of constant curvature},
        date={1963},
        ISSN={0041-6932},
     journal={Rev. Un. Mat. Argentina},
      volume={21},
       pages={131\ndash 137 (1963)},
      review={\MR{169170}},
}

\bib{santalo2009}{book}{
      author={Santal\'{o}, Luis~Antonio},
       title={Luis {A}ntonio {S}antal\'{o} selected works},
   publisher={Springer-Verlag, Berlin},
        date={2009},
        ISBN={978-3-540-89580-0},
        note={Edited by Antonio M. Naveira and Agust\'{\i} Revent\'{o}s in
  collaboration with Graciela S. Birman and Ximo Gual, With a preface by Simon
  K. Donaldson},
      review={\MR{2547470}},
}

\bib{scheuer2020}{article}{
      author={Scheuer, Julian},
       title={Minkowski inequalities and constrained inverse curvature flows in
  warped spaces},
        date={2020},
     journal={Advances in Calculus of Variations},
       pages={000010151520200050},
         url={https://doi.org/10.1515/acv-2020-0050},
}

\bib{scheuer2021}{article}{
      author={Scheuer, Julian},
       title={The {M}inkowski inequality in de {S}itter space},
        date={2021},
        ISSN={0030-8730},
     journal={Pacific J. Math.},
      volume={314},
      number={2},
       pages={425\ndash 449},
         url={https://doi.org/10.2140/pjm.2021.314.425},
      review={\MR{4337470}},
}

\bib{schneider2014}{book}{
      author={Schneider, Rolf},
       title={Convex bodies: the {B}runn-{M}inkowski theory},
     edition={expanded},
      series={Encyclopedia of Mathematics and its Applications},
   publisher={Cambridge University Press, Cambridge},
        date={2014},
      volume={151},
        ISBN={978-1-107-60101-7},
      review={\MR{3155183}},
}

\bib{schulze2008}{article}{
      author={Schulze, Felix},
       title={Nonlinear evolution by mean curvature and isoperimetric
  inequalities},
        date={2008},
        ISSN={0022-040X},
     journal={J. Differential Geom.},
      volume={79},
      number={2},
       pages={197\ndash 241},
  url={http://projecteuclid.org.prx.library.gatech.edu/euclid.jdg/1211512640},
      review={\MR{2420018}},
}

\bib{shenfeld-vanhandel2019}{article}{
      author={Shenfeld, Yair},
      author={van Handel, Ramon},
       title={Mixed volumes and the {B}ochner method},
        date={2019},
        ISSN={0002-9939},
     journal={Proc. Amer. Math. Soc.},
      volume={147},
      number={12},
       pages={5385\ndash 5402},
         url={https://doi-org.prx.library.gatech.edu/10.1090/proc/14651},
      review={\MR{4021097}},
}

\bib{shenfeld-vanhandel2022}{article}{
      author={Shenfeld, Yair},
      author={van Handel, Ramon},
       title={The extremals of {M}inkowski's quadratic inequality},
        date={2022},
        ISSN={0012-7094},
     journal={Duke Math. J.},
      volume={171},
      number={4},
       pages={957\ndash 1027},
  url={https://doi-org.prx.library.gatech.edu/10.1215/00127094-2021-0033},
      review={\MR{4393790}},
}

\bib{thale2008}{article}{
      author={Th{\"a}le, Christoph},
       title={50 years sets with positive reach---a survey},
        date={2008},
        ISSN={1843-7265},
     journal={Surv. Math. Appl.},
      volume={3},
       pages={123\ndash 165},
      review={\MR{MR2443192 (2009m:28017)}},
}

\bib{wang-xia2014}{article}{
      author={Wang, Guofang},
      author={Xia, Chao},
       title={Isoperimetric type problems and {A}lexandrov-{F}enchel type
  inequalities in the hyperbolic space},
        date={2014},
        ISSN={0001-8708},
     journal={Adv. Math.},
      volume={259},
       pages={532\ndash 556},
  url={https://doi-org.prx.library.gatech.edu/10.1016/j.aim.2014.01.024},
      review={\MR{3197666}},
}

\bib{wei-xiong2015}{article}{
      author={Wei, Yong},
      author={Xiong, Changwei},
       title={Inequalities of {A}lexandrov-{F}enchel type for convex
  hypersurfaces in hyperbolic space and in the sphere},
        date={2015},
        ISSN={0030-8730},
     journal={Pacific J. Math.},
      volume={277},
      number={1},
       pages={219\ndash 239},
         url={https://doi-org.prx.library.gatech.edu/10.2140/pjm.2015.277.219},
      review={\MR{3393689}},
}

\bib{weil1926}{article}{
      author={Weil, Andr{\'e}},
       title={Sur les surfaces a courbure negative},
        date={1926},
     journal={CR Acad. Sci. Paris},
      volume={182},
      number={2},
       pages={1069\ndash 71},
}

\bib{xu2010}{book}{
      author={Xu, Guoyi},
       title={Harmonic mean curvature flow in {R}iemannian manifolds and
  {R}icci flow on noncompact manifolds},
   publisher={ProQuest LLC, Ann Arbor, MI},
        date={2010},
        ISBN={978-1124-04603-7},
  url={http://gateway.proquest.com.prx.library.gatech.edu/openurl?url_ver=Z39.88-2004&rft_val_fmt=info:ofi/fmt:kev:mtx:dissertation&res_dat=xri:pqdiss&rft_dat=xri:pqdiss:3408443},
        note={Thesis (Ph.D.)--University of Minnesota},
      review={\MR{2941371}},
}

\end{biblist}
\end{bibdiv}

\end{document}